\documentclass[english, 12pt]{article}
\usepackage[T1]{fontenc}
\usepackage[ansinew]{inputenc}
\usepackage{amsmath}
\usepackage{amssymb}
\usepackage{verbatim}
\usepackage{enumerate}
\usepackage{setspace}
\usepackage{tikz}
\usepackage{abstract}
\usepackage[nottoc]{tocbibind}

\usepackage[USenglish]{babel}
\usepackage{amstext}\usepackage{amsthm}\usepackage{color}

\newcommand{\N}{\mathbb{N}}
\newcommand{\Z}{\mathbb{Z}}

\usepackage{babel}
\usepackage{cleveref}

\begin{document}

\newtheorem{thm}{Theorem}[section] \crefformat{thm}{#2Theorem~#1#3}
\newtheorem{cor}[thm]{Corollary} \crefformat{cor}{#2Corollary~#1#3} 
\newtheorem{lem}[thm]{Lemma} \crefformat{lem}{#2Lemma~#1#3} 
\newtheorem{prop}[thm]{Proposition} \crefformat{prop}{#2Proposition~#1#3} 
\theoremstyle{definition} \newtheorem{defin}[thm]{Definition} \crefformat{defin}{#2Definition~#1#3} 
\theoremstyle{remark} \newtheorem{example}[thm]{Example} \crefformat{example}{#2Example~#1#3} 
\theoremstyle{remark} \newtheorem{rem}[thm]{Remark} \crefformat{rem}{#2Remark~#1#3} 
\newtheorem{prob}[thm]{Problem} \crefformat{prob}{#2Problem~#1#3} 
\crefformat{subsection}{#2subsection~#1#3}

\pagestyle{plain}

\title{\bf{\large Non-Solvable Graph of a Finite Group and Solvabilizers}}
\author{\normalsize Doron Hai-Reuven\makeatletter\footnote{Department of Mathematics, Bar-Ilan University, Israel. \newline Email: {\tt doron.hai@live.biu.ac.il}. \newline \indent This paper is part of the author's MSc thesis, which was written at Bar-Ilan University under the supervision of Prof. B.\ Kunyavskii.}\makeatother}
\date{}

\maketitle

\begin{abstract}
Let $G$ be a finite group. For $x \in G$, we define the solvabilizer of $x$ in $G$, denoted $sol_G(x)$, to be the set $\{ g \in G \mid \langle g,x \rangle$ is solvable$\}$. A group $G$ is an S-group if $sol_G(x)$ is a subgroup of $G$ for every $x \in G$. In this paper we prove that $G$ is solvable $\Leftrightarrow$ $G$ is an S-group. Secondly, we define the non-solvable graph of $G$ (denoted ${\mathcal S}_{G}$). Its vertices are $G$ and there is an edge between $x,y \in G$ whenever $\langle x,y \rangle$ is not solvable. If $S(G)$ is the solvable radical of $G$ and $G$ is not solvable, we look at the induced graph over $G \setminus S(G)$, denoted $\widehat{{\mathcal S}_{G}}$. We prove that if $G$ is not solvable, then $\widehat{{\mathcal S}_{G}}$ is irregular. In addition, we prove some properties of solvabilizers and non-solvable graphs.
\end{abstract}

\newpage

\section{Introduction} \label{Introduction}
\pagenumbering{arabic}
\setcounter{page}{2} 

One of the most interesting approaches in the study of groups is to associate a graph to each group. One can learn
about the properties of a group by exploring its graph. If $G$ is a finite group (throughout this paper we will assume $G$ is finite) and $R \subseteq G \times G$ is a relation over $G$, then we can associate a graph $(V, E)$ to $G$ as follows: Take $V = G$ as vertices and draw an edge between $x,y \in G$ if and only if $xRy$. Surprisingly, the information we get by looking only at relations between pairs of elements in the group, is sometimes sufficient for concluding that the whole group has a certain property. This field is relatively new, and over the years different types of graphs were defined, such as:

\noindent
\begin{enumerate}
\item {} Non-Commuting graph: For the relation $R = \{(x,y) \mid [x,y] \neq e_G \}$. Some results for this graph are described in~\cite{NCG}.

\item {} Non-Nilpotent graph: For the relation $R = \{(x,y) \mid \langle x,y \rangle$ is not nilpotent$\}$. Some research on this graph is made in~\cite{NNGG}.

\item {} Prime graph: This example is more famous and very different from the last two examples. As described in~\cite{AVV}, the vertices are the set of prime numbers which divide $|G|$, and there is an edge between two distinct vertices $p,q$ if there exists $x \in G$ with $O(x) = pq$. 
\end{enumerate} 

\pagebreak
\noindent
We want to define a graph which will help us to explore the property of being solvable. One of the most famous and deep results on solvable groups, Thompson's theorem, states that a group $G$
is solvable if and only if $\langle x,y \rangle$ is solvable for every $x,y \in G$. If $R = \{(x,y) \in G \times G \mid \langle x,y \rangle$ is not solvable$\}$ and the graph $(V_G, E_G)$ is associated to $G$ by the relation $R$, then an equivalent formulation for Thompson's theorem will be: $G$ is solvable if and only if $(V_G, E_G)$ is an empty graph. A new result, described in~\cite{CAG}, leads to a characterization of finite simple non-abelian groups by graphs. These results (and many others) are demonstrating the importance of the research of groups via graphs.
When doing such a research, it is important to learn as much as we can about the graph properties. Each property of the graph can teach us about a property of the group. In this paper we present the non-solvable graph of a group, along with special subsets of a group - solvabilizers. Given a group $G$, the non-solvable graph is defined by the relation $R$ above. If $x \in G$, then the solvabilizer of $x$ is $G \setminus Adj(x)$, where $Adj(x)$ is the set of the neighbours of $x$ in the non-solvable graph of $G$. The idea of defining such a graph and such subsets is natural, given the definitions in~\cite{NCG} (non-commuting graph and centralizers) and~\cite{NNGG} (non-nilpotent graph and nilpotentizers). 
\newline \newline
\noindent
In Section 2 we discuss solvabilizers. We present some basic properties of solvabilizers and prove that for every group $G$, the size of the centralizer of $x \in G$ divides the size of the solvabilizer of $x$ (Proposition 2.13). We define a new type of group, which we call an S-group. A group $G$ is an S-group if the solvabilizer of every $x \in G$ is a subgroup of $G$. We prove that $G$ is solvable if and only if $G$ is an S-group (Proposition 2.22). Equivalently, this proves that $G$ is solvable if and only if $G$ has the following property: For every $a,b,x \in G$, if $\langle a,x \rangle$ and $\langle b,x \rangle$ are solvable, then $\langle ab,x \rangle$ is solvable. This result is a new, equivalent condition for solvability. In addition, we note which of the properties we proved is relevant also for nilpotentizers.
\newline \newline
\noindent
In Section 3 we present the non-solvable graph of a group. We prove some interesting features of the graph, one of them is that if $S(G)$ is the solvable radical of $G$, then the induced graph over $G \setminus S(G)$ is irregular (Corollary 3.17). Dealing with the property of being irregular is important, as shown in~\cite{NNGG}. In addition, we note which of the properties we proved is relevant also for the non-nilpotent graph.
\newpage
\section{Solvabilizers}
\begin{defin}
Let $A,B \subseteq G$ be two non-empty subsets. The {\bf solvabilizer} of $B$ with respect to $A$, denoted by $Sol_{A}(B)$, is the subset 
\[
\{a \in A \mid \langle a,b \rangle \text{ is solvable } \forall b \in B \}
\]
For empty subsets, we define $Sol_{\emptyset}(B) := \emptyset$, $Sol_{A}(\emptyset) := A$. Finally, for brevity, we define $Sol_{A}(x) := Sol_{A}(\{x\})$, $Sol(G) := Sol_{G}(G)$.
\end{defin}
\noindent
The last definition is a natural extension to an earlier definition of a special type of subset of $G$, called ``nilpotentizer''.

\begin{defin} (~\cite{NNGG} )
Let $A,B \subseteq G$ be two non-empty subsets. The {\bf nilpotentizer} of $B$ with respect to $A$, denoted by $nil_{A}(B)$, is defined similarly to $sol_{A}(B)$ by replacing the ``solvable'' condition with ``nilpotent''.
\end{defin}

\begin{rem}
$Sol_{A}(B)$ need not be a subgroup of $G$ in general, even in the case $A,B \leq G$. Obviously,
$Sol_{A}(x) = Sol_{A}(\langle x \rangle)$ and by GAP~\cite{GAP} we get $|Sol_{A_{5}}((1,2,3))| = 24 \nmid 60 = |A_{5}|$. \newline So $Sol_{A_{5}}((1,2,3)) = Sol_{A_{5}}(\langle (1,2,3) \rangle) \nleq A_{5}$.
\end{rem}

\begin{thm} \label{thmTLC} (~\cite{TLC} )
$Sol(G) = S(G)$.
\end{thm}

\begin{rem}
Note that $e_{G} \in A \subseteq G \Rightarrow e_{G} \in Sol_{A}(B)$ for every set $B$. Also, 
$x \in Sol_{A}(B)$, $x^{k} \in A$ for some $k \in \Z \Rightarrow x^{k} \in Sol_{A}(B)$. So if we look at a solvabilizer in the form $sol_G(x)$ for some $x \in G$, then $sol_G(x) \leq G$ if and only if $sol_G(x)$ is closed under multiplication in $G$. Another interesting fact, easily followed by Theorem 2.4, is that $G$ is solvable $\Leftrightarrow Sol(G) = G$.
\end{rem}
\noindent
We will use the notation $Sol(G)$ instead of $S(G)$, in order to keep in mind the original definition of $Sol(G)$.
At first, we will present some basic properties of solvabilizers.

\begin{lem}
Let $G$ be a group, $H \leq G$ a solvable subgroup and $s \in Sol(G)$. Then $\langle H, \{s\} \rangle$ is solvable.
\end{lem}

\begin{proof}
Define $K = HSol(G)$. $Sol(G) \vartriangleleft G$ and therefore $K \leq G$. \newline $Sol(G) \vartriangleleft K$ and by the second isomorphism theorem we get $K/Sol(G) \cong H/H \cap Sol(G)$. $Sol(G)$ is solvable and $K/Sol(G)$ is 
solvable (isomorphic to a quotient group of the solvable group $H$). Therefore $K$ is solvable and 
$\langle H, \{s\} \rangle \leq K$ is solvable as well.
\end{proof}

\begin{lem}
Let $G$ be a group. Then $Sol(G)Sol_G(x) = Sol_G(x)$ $\forall x \in G$.
\end{lem}

\begin{proof}
Obviously, $Sol(G)Sol_G(x) \supseteq Sol_G(x)$. \newline Let $s \in Sol(G)$ and $a \in Sol_G(x)$. $\langle a, x \rangle$ is solvable and by Lemma 2.6 $\langle a, x, s \rangle$ is solvable. $\langle sa, x \rangle \leq \langle a, x, s \rangle$, so  $\langle sa, x \rangle$ is solvable. Therefore $sa \in Sol_G(x)$ and $Sol(G)Sol_G(x) \subseteq Sol_G(x)$.
\end{proof}

\begin{lem}
Let $G$ be a group. Then $|Sol(G)| \mid |Sol_G(x)|$ $\forall x \in G$.
\end{lem}

\begin{proof}
By Lemma 2.7, if $s \in Sol(G)$ and $a \in Sol_G(x)$, $H = Sol(G)$ acts on the set $A = Sol_G(x)$ by $s*a := sa$.
$H_{a} = \{h \in H \mid ha = a \} = \{e_{H}\}$. $|H_{a}| = 1$ $\forall a \in A \Rightarrow |H| \mid |A|$.
\end{proof}

\begin{lem}
Let $A,B,C \subseteq G$ be three subsets. Then we have:
\begin{enumerate}
\item $A \subseteq B \Rightarrow Sol_{A}(C) \subseteq Sol_{B}(C)$, $Sol_{C}(B) \subseteq Sol_{C}(A)$.
\item $Sol_{A}(Sol_{B}(A)) = A$.
\item $A \subseteq B \Rightarrow Sol_{A}(C) = A \cap Sol_{B}(C)$.
\item $Sol_{C}(A \cup B) = Sol_{C}(A) \cap Sol_{C}(B)$, $Sol_{C}(A \cap B) \supseteq Sol_{C}(A) \cup Sol_{C}(B)$.
\item $Sol_{A}(B) = \bigcap_{x \in B}Sol_{A}(x)$. Particularly, $Sol(G) = \bigcap_{x \in G}Sol_{G}(x)$.
\end{enumerate}
\end{lem}

\begin{proof} 
It is straightforward.
\end{proof}

\begin{lem}
If $Sol_{G}(x) \leq G$ $\forall x \in G$ then $Sol_{H}(A) \leq G$ $\forall A \subseteq G, H \leq G$.
\end{lem}

\begin{proof}
\[
Sol_{H}(A) = H \cap Sol_{G}(A) = H \cap \bigcap_{x \in A} Sol_{G}(x)
\]
$Sol_{H}(A)$ is an intersection of subgroups
and therefore a subgroup.
\end{proof}
\noindent
By translating~\cite{NNGG} into solvabilizers terms we gain more properties, as described in the next two lemmas.

\begin{lem}
Let $G$ be a group, $N \vartriangleleft G$, $N \subseteq Sol(G)$ and $x,y,g \in G$. Then we have:
\begin{enumerate}
\item $\langle x \rangle = \langle y \rangle \Rightarrow Sol_{G}(x)=Sol_{G}(y)$.
\item $Sol_{G}(gxg^{-1}) = gSol_{G}(x)g^{-1}$.
\item $Sol_{G/N}(xN) = Sol_G(x)/N$.
\end{enumerate}
\end{lem}

\begin{proof} 

\hspace*{\fill}

\begin{enumerate}
\item $\langle x \rangle = \langle y \rangle \Rightarrow x \in \langle y \rangle$, 
$y \in \langle x \rangle \Rightarrow \langle a,x \rangle \subseteq \langle a,y \rangle$, $\langle a,y \rangle \subseteq \langle a,x \rangle$ for all $a \in G \Rightarrow \langle a,x \rangle = \langle a,y \rangle$ for all $a \in G \Rightarrow Sol_{G}(x)=Sol_{G}(y)$.

\item $Sol_{G}(gxg^{-1}) = \newline \{y \in G \mid \langle y,gxg^{-1} \rangle$ is solvable$\} = \newline
\{gyg^{-1} \in G \mid \langle gyg^{-1},gxg^{-1} \rangle$ is solvable$\} = \newline
\{gyg^{-1} \in G \mid g \langle y,x \rangle g^{-1}$ is solvable $\}$. \newline
$g \langle y,x \rangle g^{-1}$ is solvable $\Leftrightarrow \langle y,x \rangle$ is solvable, so \newline
$Sol_{G}(gxg^{-1}) = \{gyg^{-1} \in G \mid \langle y,x \rangle$ is solvable$\} = gSol_{G}(x)g^{-1}$. \newline Note that this fact implies that $|Sol_G(x)|$ is constant on conjugacy classes.

\item $N$ is solvable and therefore $\exists l \in \N$ such that $N^{(l)} = \{e_G\}$. 
\newline
$Sol_{G/N}(xN) = \newline \{gN \mid \langle gN,xN \rangle$ is solvable$\} = \{ gN \mid \exists k \in \N$ $\langle g,x \rangle^{(k)} \subseteq N \}  = \newline \{ gN \mid \exists k \in \N$ $\langle g,x \rangle^{(k+l)} \subseteq N^{(l)} = \{e_G\} \} = \{gN \mid \langle g,x \rangle$ is solvable$\} = Sol_G(x)/N$.

\end{enumerate}
\end{proof}

\begin{lem}
Let $G$ be a group. Then $O(x) \mid |Sol_{G}(x)|$ $\forall x \in G$.
\end{lem}

\begin{proof}
Let $x \in G$ and $A = \{H \leq G \mid x \in H$ and $H$ is solvable$\}$. $\langle x \rangle \in A$, so $A$ is not empty. 
Let $Y = \bigcup A$. If $y \in Sol_{G}(x)$ then $\langle y,x \rangle$ is solvable and therefore 
$\langle y,x \rangle \in A \Rightarrow y \in Y$. If $y \in Y$ then $y \in H$ for some solvable subgroup $H$ with
$x \in H$. $\langle x,y \rangle \leq H \Rightarrow \langle x,y \rangle$ is solvable $\Rightarrow y \in Sol_{G}(x)$.
So $Sol_{G}(x) = Y$, which is the union of all solvable subgroups containing $x$. Each of those groups is a disjoint
union of some cosets of $\langle x \rangle$. Therefore, $Sol_{G}(x)$ is a disjoint union of some cosets of
$\langle x \rangle \Rightarrow O(x) \mid |Sol_{G}(x)|$.
\end{proof}
\noindent
Now we will prove a stronger property of solvabilizers.

\begin{prop}
Let $G$ be a group. Then $|C_G(x)| \mid |Sol_{G}(x)|$ $\forall x \in G$.
\end{prop}

\begin{rem}
In general, $Sol_G(x) \varsubsetneq C_G(x)Sol_G(x)$. By GAP~\cite{GAP} we get \newline $|Sol_{A_5}((2,3)(4,5))| = 36$ while $C_{A_5}((2,3)(4,5))Sol_{A_5}((2,3)(4,5)) = A_5$.
\end{rem}

\begin{proof}
We use induction on $|G|$. If $|G| = 1$ then $|C_G(x)| = |Sol_{G}(x)| = 1$ $\forall x \in G$. Let's assume the statement holds for every group $G$ which satisfies $|G| < m$ for a certain constant $m \in \N$. Let $G$ be a group with $|G| = m$. 
\newline
If $|Z(G)| > 1$ then the statement holds for $G/Z(G)$. 
\newline
Therefore $|C_{G/Z(G)}(xZ(G))| \mid |Sol_{G/Z(G)}(xZ(G))|$ $\forall x \in G$. 
Note that for every group $G$, $C_G(x)/Z(G) \leq C_{G/Z(G)}(xZ(G))$ - Indeed, 
\[
C_G(x)/Z(G) = \{aZ(G) \mid [a,x] = e_G \} \subseteq \{aZ(G) \mid [a,x] \in Z(G) \} = C_{G/Z(G)}(xZ(G))
\]
Also, by Lemma 2.11, $|Sol_G(x)/Z(G)| = \frac{|Sol_G(x)|}{|Z(G)|}$.
So on the one hand,
\[
\frac{|C_G(x)|}{|Z(G)|} = |C_G(x)/Z(G)| \mid |C_{G/Z(G)}(xZ(G))|
\]
On the other hand, 
\[
|C_{G/Z(G)}(xZ(G))| \mid |Sol_{G/Z(G)}(xZ(G))| = |Sol_G(x)/Z(G)| = \frac{|Sol_G(x)|}{|Z(G)|}
\]
Therefore,
\[
\frac{|C_G(x)|}{|Z(G)|} \mid |C_{G/Z(G)}(xZ(G))| \mid \frac{|Sol_G(x)|}{|Z(G)|} \Rightarrow 
\]
\[
\frac{|C_G(x)|}{|Z(G)|} \mid \frac{|Sol_G(x)|}{|Z(G)|} \Rightarrow
\]
\[ 
|C_G(x)| \mid |Sol_G(x)|
\]
We are left with the case $|Z(G)| = 1$. Let us denote $H = C_G(x)$, $A = Sol_G(x)$. If $h \in H$ and $a \in A$, $H$ acts on $A$ by $h*a = hah^{-1}$. Indeed, $a \in A \Leftrightarrow \langle a,x \rangle$ is solvable $\Leftrightarrow 
h\langle a,x \rangle h^{-1}$ is solvable $\Leftrightarrow \langle hah^{-1},hxh^{-1} \rangle = \langle hah^{-1},x \rangle$ is solvable $\Leftrightarrow hah^{-1} \in A$. 
\newline
So for $H_{a} = \{h \in H \mid hah^{-1} = a\} = H \cap C_G(a) = C_G(x) \cap C_G(a)$ and for 
$A_{h} = \{a \in A \mid hah^{-1} = a \} = A \cap C_G(h) = Sol_G(x) \cap C_G(h)$, we get 
\[
\sum\limits_{a \in Sol_G(x)}^{}|C_G(x) \cap C_G(a)| = \sum\limits_{h \in C_G(x)}^{}|Sol_G(x) \cap C_G(h)|
\]
In a different use of notation, the last equation can be shown as
\[
(*) \sum\limits_{a \in Sol_G(x)}^{}|C_{C_G(x)}(a)| = \sum\limits_{a \in C_G(x)}^{}|Sol_{C_G(a)}(x)|
\]
If $a \in Sol_G(x)$ then $conj_{C_G(x)}(a) \subseteq Sol_G(x)$. So $Sol_G(x)$ is a disjoint union of subsets in the form $conj_{C_G(x)}(a)$. Also, for every group $G$ and $x,g \in G$, \newline $gC_G(x)g^{-1} = \{gyg^{-1} \mid [y,x] = e_G \} = \{ gyg^{-1} \mid g[y,x]g^{-1} = e_G \} = \newline \{ gyg^{-1} \mid [gyg^{-1},gxg^{-1}] = e_G \} = \{ y \mid [y,gxg^{-1}] = e_G \} = C_G(gxg^{-1})$, \newline so $|C_G(x)|$ is constant on conjugacy classes.
\newline
Therefore, there exist $k \in \N$ and $a_1, a_2, ... , a_k \in Sol_G(x)$ such that
\newline \newline
$\sum\limits_{a \in Sol_G(x)}^{}|C_{C_G(x)}(a)| = \sum\limits_{i = 1}^{k}|conj_{C_G(x)}(a_i)||C_{C_G(x)}(a_i)| = 
\sum\limits_{i = 1}^{k}|C_G(x)| = k|C_G(x)|$.
\newline \newline
$Z(G) = 1$ and therefore $C_G(a) = G \Leftrightarrow a = e_G$. For $a \neq e_G$, $|C_G(a)| < |G| = m$. If $a \in C_G(x)$ then $x \in C_G(a)$ and by induction \newline $|C_{C_G(a)}(x)| \mid |Sol_{C_G(a)}(x)|$. Let us define 
$n(a, x) = \frac{|Sol_{C_G(a)}(x)|}{|C_{C_G(a)}(x)|}$. By induction, $n(a, x) \in \N$ $\forall a \in C_G(x) \setminus \{e_G\},$ \newline $x \in C_G(a)$.
Let $y \in C_G(x)$. $n(yay^{-1}, x) = \frac{|Sol_{C_G(yay^{-1})}(x)|}{|C_{C_G(yay^{-1})}(x)|} = 
\frac{|Sol_{C_G(a)}(y^{-1}xy)|}{|C_{C_G(x)}(yay^{-1})|} = \frac{|Sol_{C_G(a)}(x)|}{|C_{C_G(x)}(a)|} = 
\frac{|Sol_{C_G(a)}(x)|}{|C_{C_G(a)}(x)|} = n(a, x)$. We get:
\[
\sum\limits_{a \in C_G(x)}^{}|Sol_{C_G(a)}(x)| = \sum\limits_{a = e_G}^{}|Sol_{C_G(a)}(x)| \text{ } + \sum\limits_{a \in C_G(x) \setminus \{e_G\} }^{}|Sol_{C_G(a)}(x)| =
\]
\[
|Sol_G(x)| \text{ } + \sum\limits_{a \in C_G(x) \setminus \{e_G\} }^{}n(a, x)|C_{C_G(a)}(x)| = |Sol_G(x)| + \text{ } \sum\limits_{a \in C_G(x) \setminus \{e_G\} }^{}n(a, x)|C_{C_G(x)}(a)|
\]
$C_G(x)$ is a disjoint union of its conjugacy classes: \newline $C_G(x) = conj_{C_G(x)}(a_1) \uplus conj_{C_G(x)}(a_2) \uplus ... \uplus conj_{C_G(x)}(a_l)$, where $l \in \N$, $a_1, a_2, ... , a_l \in C_G(x)$ are representatives of the $l$ conjugacy classes 
\newpage
\noindent
and $a_l = e_G$. If $l = 1$, then $G = \{e_G\}$ and the proposition holds. If $l > 1$ we get:
\[
|Sol_G(x)| \text{ } + \sum\limits_{a \in C_G(x) \setminus \{e_G\} }^{}n(a, x)|C_{C_G(x)}(a)| = 
\]
\[
|Sol_G(x)| \text{ } + \text{ } \sum\limits_{i = 1}^{l-1} \sum\limits_{a \in conj_{C_G(x)}(a_i)}^{} n(a, x)|C_{C_G(x)}(a)| 
\]
\noindent
$n(a, x) = n(a_i, x)$ and $|C_{C_G(x)}(a)| = |C_{C_G(x)}(a_i)|$ for every $a \in conj_{C_G(x)}(a_i)$ and $1 \leq i < l$. Therefore, 
\[
|Sol_G(x)| \text{ } + \sum\limits_{i = 1}^{l-1} \sum\limits_{a \in conj_{C_G(x)}(a_i)}^{} n(a, x)|C_{C_G(x)}(a)| = 
\]
\[
|Sol_G(x)| \text{ } + \sum\limits_{i = 1}^{l-1} \sum\limits_{a \in conj_{C_G(x)}(a_i)}^{} n(a_i, x)|C_{C_G(x)}(a_i)| = 
\]
\[
|Sol_G(x)| \text{ } + \sum\limits_{i = 1}^{l-1} n(a_i, x)|C_{C_G(x)}(a_i)| \sum\limits_{a \in conj_{C_G(x)}(a_i)}^{} 1 = 
\]
\[
|Sol_G(x)| \text{ } + \sum\limits_{i = 1}^{l-1} n(a_i, x)|C_{C_G(x)}(a_i)||conj_{C_G(x)}(a_i)| = 
\]
\[
|Sol_G(x)| \text{ } + \sum\limits_{i = 1}^{l-1} n(a_i, x)|C_G(x)| = 
\]
\[
|Sol_G(x)| \text{ } + |C_G(x)|\sum\limits_{i = 1}^{l-1} n(a_i, x)
\]
$n(a_i, x) \in \N$ $\forall$ $1 \leq i \leq l-1$, so $\sum\limits_{i = 1}^{l-1} n(a_i, x) = t \in \N$. 
\newline \newline
If we return to equation (*) we get:
\[
k|C_G(x)| = |Sol_G(x)| \text{ } + t|C_G(x)|, \text{ } t, k \in \N \text{ } \Rightarrow |C_G(x)| \mid |Sol_G(x)|
\] 
\end{proof}

\begin{rem}
If we replace $Sol_G(x)$ with $Nil_G(x)$ in the last proof, is stays valid. So $|C_G(x)| \mid |Nil_G(x)|$ for every group $G$ and $x \in G$.
\end{rem}

\begin{prop}
Let $G$ be a group. Then $|G|$ divides $\sum\limits_{x \in G}^{}|Sol_G(x)|$.
\end{prop}

\begin{proof}
Suppose that $G$ has $k$ conjugacy classes for some $k \in \N$ and $a_1, a_2, ... , a_k$ are representatives of the classes. By Lemma 2.11, $|Sol_G(x)|$ is constant on every conjugacy class. By Proposition 2.13 $|C_G(x)| \mid |Sol_G(x)|$ for every $x \in G$. 
Therefore there exist $n_1, n_2, ... , n_k \in \N$ such that
\[
\sum\limits_{x \in G}^{}|Sol_G(x)| = \sum\limits_{i = 1}^{k}|conj_G(a_i)||Sol_G(a_i)| =
\]
\[
\sum\limits_{i = 1}^{k}|conj_G(a_i)||C_G(a_i)|n_i = \sum\limits_{i = 1}^{k}|G|n_i = |G|\sum\limits_{i = 1}^{k}n_i \Rightarrow
\]
\[
|G| \text{ divides } \sum\limits_{x \in G}^{}|Sol_G(x)|
\]
\end{proof}

\newpage
\noindent

\begin{rem}
Again, the last proof is valid for nilpotentizers. So $|G|$ divides $\sum\limits_{x \in G}^{}|Nil_G(x)|$ for every group $G$.
\end{rem}

\begin{defin}
A group $G$ is called an {\bf S-group} if $Sol_G(x) \leq G$ $\forall x \in G$.
\end{defin}

\noindent
By Lemma 2.10, if $G$ is an S-group and $H \leq G$  then $H$ is an S-group. Also, note that if $G$ is solvable then $\langle x,y \rangle$ is solvable $\forall x,y \in G \Rightarrow \forall x \in G$ $Sol_G(x) = G \leq G$. So if $G$ is solvable then $G$ is an S-group. Now, in few simple steps, we will prove that the opposite is also true.

\begin{lem}
Let $G$ be a group and $x,y \in G$ such that $O(x) = O(y) = 2$. Then $\langle x, y \rangle$ is solvable.
\end{lem}

\begin{proof}
$G$ is dihedral and hence solvable.
\end{proof}

\begin{lem}
Let $G$ be a group and let $N \vartriangleleft G$ such that $N \subseteq Sol(G)$. Then $G$ is an S-group $\Leftrightarrow G/N$ is an S-group.
\end{lem}

\begin{proof}
Let $x \in G$. $Sol_{G/N}(xN) = Sol_G(x)/N$, so if $G$ is an S-group then so is $G/N$. Suppose $G/N$ is an S-group. If $a,b \in Sol_G(x)$ then $\langle a,x \rangle$, $\langle b,x \rangle$ are solvable $\Rightarrow \langle aN,xN \rangle$, $\langle bN,xN \rangle$ are solvable $\Rightarrow \langle abN,xN \rangle$ is solvable $\Rightarrow \langle ab,x \rangle ^ {(k)} \subseteq N$ for some $k \in \N \Rightarrow \langle ab,x \rangle ^ {(k+l)} \subseteq N^{(l)} = \{e_G\}$ for some $l \in \N \Rightarrow \langle ab,x \rangle$ is solvable $\Rightarrow G$ is an S-group.
\end{proof}

\begin{prop}
Let $G$ be a simple S-group. Then $G$ is abelian.
\end{prop}

\begin{proof}
Let $G$ be a simple S-group. If $|G|$ is odd, then by Feit-Thompson theorem (~\cite{ROB} 5.4 ) $G$ is solvable. If $|G|$ is even, then by Cauchy's theorem there exists $a \in G$ such that $O(a) = 2$. By Lemma 2.11
\[
Core_G(Sol_G(a)) = \bigcap_{g \in G}(gSol_G(a)g^{-1}) = \bigcap_{g \in G}(Sol_G(gag^{-1})) \vartriangleleft G
\]
$O(a) = O(gag^{-1}) = 2$ $\forall g \in G$, so by Lemma 2.19 $a \in Core_G(Sol_G(a))$. $G$ is simple and $a \in Core_G(Sol_G(a)) \vartriangleleft G$, so $Core_G(Sol_G(a)) = G \Rightarrow$ \newline $Sol_G(a) = G \Rightarrow a \in Sol(G) \vartriangleleft G \Rightarrow Sol(G) = G \Rightarrow G$ is solvable. $G$ is simple and solvable $\Rightarrow$ $G$ is abelian.
\end{proof}

\begin{prop}
Let $G$ be a group. Then $G$ is solvable $\Leftrightarrow G$ is an S-group.
\end{prop}

\begin{proof}
Assume that there exists at least one non-solvable S-group. Among those groups, denote $G$ as the smallest one (i.e., if $K$ is a non-solvable S-group then $|G| \leq |K|$). If $G$ is simple, then by the last proposition $G$ is solvable, a contradiction. If $G$ is not simple, then there exists $N \vartriangleleft G$, $1 < |N| < |G|$. So $N$ is a smaller S-group and therefore solvable $\Rightarrow N \subseteq Sol(G)$. \newline By Lemma 2.20, $G/N$ is an S-group. $|G/N| < |G|$, so $G/N$ is also solvable. $N, G/N$ are solvable $\Rightarrow G$ is solvable, a contradiction. Therefore, there are no non-solvable S-groups $\Rightarrow$ every S-group is solvable. 
\end{proof}
\noindent
The last result is quite strong. It is a new, equivalent condition for solvability. We can write this result in a different way: If for every $a,b,x \in G$ the property ``$\langle a,x \rangle, \langle b,x \rangle$ are solvable $\Rightarrow \langle ab,x \rangle$ is solvable'' holds, then $G$ is solvable. It almost seems that we obtained this result too easy, but it's not the case. Note that we used the Feit-Thompson theorem and Theorem 2.4 in the proof.

\section{Non-Solvable Graph of a Group}

\begin{defin}
Let $G$ be a group. {\bf The non-solvable graph} of $G$, denoted by ${\mathcal S}_{G}$, is a simple graph with group elements as vertices, such that $(x,y)$ is an edge $\Leftrightarrow \langle x,y \rangle$ is not solvable.
\end{defin}
\noindent
By Thompson's theorem, ${\mathcal S}_{G}$ is an empty graph $\Leftrightarrow G$ is solvable. Therefore, 
${\mathcal S}_{G}$ is interesting only if $G$ is not solvable. It is clear that $Sol(G)$ elements are exactly the isolated vertices in ${\mathcal S}_{G}$. Thus, if $G$ is not solvable, it is natural to choose to explore the induced graph of ${\mathcal S}_{G}$ with respect to $G \setminus Sol(G)$, which will be denoted $\widehat{{\mathcal S}_{G}}$.
Note that the degree of a vertex $x$ in ${\mathcal S}_{G}$ is equal to its degree in $\widehat{{\mathcal S}_{G}}$. Also, since every vertex in $\widehat{{\mathcal S}_{G}}$ is taken from $G \setminus Sol(G)$, all vertices have an order greater than 1.
\newline
The idea of building a graph from a group is not new, as shown in the next definition.
\begin{defin}(~\cite{NNGG} )
Let $G$ be a group. {\bf The non-nilpotent graph} of $G$, denoted by ${\mathcal N}_{G}$, is a simple graph with group elements as vertices, such that $(x,y)$ is an edge $\Leftrightarrow \langle x,y \rangle$ is not nilpotent.
\end{defin}
\begin{rem}
${\mathcal S}_{G}$ is obviously a subgraph of ${\mathcal N}_{G}$. $nil(G) = Z^{*}(G)$ (~\cite{NNGG} ) and therefore $\widehat{{\mathcal S}_{G}}$ is a subgraph of $\widehat{{\mathcal N}_{G}}$.
\end{rem}

\begin{thm} (~\cite{TLC} 6.4)
Let $G$ be a non-solvable group. Suppose $x,y \in G$ such that $x,y \notin Sol(G)$. Then there exists $s \in G$ such that $\langle x,s \rangle$, $\langle y,s \rangle$ are not solvable.
\end{thm}

\noindent
In terms of $\widehat{{\mathcal S}_{G}}$, the latter theorem states that $\widehat{{\mathcal S}_{G}}$
is connected and its diameter is at most 2 (for a non-solvable group).

\begin{lem}
Let $G$ be a non-solvable group. Then $diam(\widehat{{\mathcal S}_{G}}) \neq 1$.
\end{lem}

\begin{proof}
We start the proof similarly to the proof of Proposition 2.1 in~\cite{NCG}: 
Suppose, for a contradiction, that $diam(\widehat{{\mathcal S}_{G}}) = 1$. \newline Let $x \notin Sol(G)$. If $x \neq x^{-1}$, then $x,x^{-1} \notin Sol(G)$ and $\langle x,x^{-1} \rangle = \langle x \rangle$ is solvable, so 
$diam(\widehat{{\mathcal S}_{G}}) \geq d(x,x^{-1}) \geq 2$ and we get a contradiction. So $x = x^{-1}$ for all 
$x \notin Sol(G)$. Now we continue the proof as follows: If $x,y \notin Sol(G)$ are distinct, then 
$diam(\widehat{{\mathcal S}_{G}}) = 1 \Rightarrow \langle x,y \rangle$ is not solvable. Therefore, 
$\langle xy,x \rangle = \langle x,y \rangle$ is not solvable, so $xy \notin Sol(G)$. Take 
$A = (G \setminus Sol(G)) \cup \{e_{G}\}$. $e_{G} \in A$ and $A$ is closed under multiplication and the inverse operation $\Rightarrow A$ is a subgroup of $G$. $O(a)=2$ $\forall a \in A \Rightarrow A$ is abelian and therefore solvable. Now, $A \cup Sol(G) = G \Rightarrow G = A$ or $G = Sol(G) \Rightarrow G$ is solvable, a contradiction.
So $diam(\widehat{{\mathcal S}_{G}}) \neq 1$ in case $G$ is a non-solvable group.
\end{proof}

\noindent
By the last lemma, we conclude that if $G$ is not solvable then $diam(\widehat{{\mathcal S}_{G}}) = 2$.

\begin{lem}
Let $G$ be a group and let $x \in G$. Then $|C_G(x)| \mid |deg(x)|$ 
\end{lem}

\begin{proof}
$Sol_G(x) \uplus Adj(x) = G \Rightarrow |Sol_G(x)| + deg(x) = |G|$. $|C_G(x)|$ divides $|G|$ and by Proposition 2.13 $|C_G(x)|$ divides $|Sol_G(x)| \Rightarrow |C_C(x)|$ divides $deg(x)$.
\end{proof}

\begin{rem}
As we saw in the last section, $|C_G(x)|$ divides $|Nil_G(x)|$. Therefore, $|C_G(x)| \mid |deg(x)|$ in the graph ${\mathcal N}_{G}$. 
\end{rem}

\begin{lem}
Let $G$ be a non-solvable group and $x$ be a vertex in $\widehat{{\mathcal S}_{G}}$. Then $O(x) < deg(x)$.
\end{lem}

\begin{proof}
$\widehat{{\mathcal S}_{G}}$ is connected and therefore there exists a vertex $y \neq x$ such that $\langle x,y \rangle$ is not solvable. Denote $O(x) = k$ and $A = \{y, xy, x^{2}y, ... , x^{k-1}y \} \cup \{yx\}$. Then $|A| = k+1$:
Indeed, If $x^{s}y = x^{t}y$ for some $0 \leq s < t \leq k-1$, then $x^{t-s} = e_G$, a contradiction to $O(x) = k$. Secondly, if $yx =  x^{t}y$ for some $0 \leq t \leq k-1$, then $yxy^{-1} = x^{t} \in \langle x \rangle \Rightarrow \langle x \rangle \vartriangleleft \langle x,y \rangle$. $\langle x \rangle$ is solvable and so is $\langle x,y \rangle / \langle x \rangle \cong \langle y \rangle$. Therefore, $\langle x, y \rangle$ is solvable, a contradiction. 
If $a \in A$, then $\langle x,a \rangle = \langle x,y \rangle$, which is not solvable. Thus, every element of $A$ is a neighbour of $x \Rightarrow deg(x) \geq k+1 > k = O(x)$.
\end{proof}

\begin{cor}
$2O(x) \leq deg(x)$ for every vertex $x$ in $\widehat{{\mathcal S}_{G}}$.
\end{cor}

\begin{proof}
Combine Lemma 3.6 and Lemma 3.8.
\end{proof}

\begin{lem}
Let $G$ be a non-solvable group and $x$ be a vertex in $\widehat{{\mathcal S}_{G}}$. Then $deg(x)$ is not prime.
\end{lem}

\begin{proof}
Assume $deg(x) = p$ for a prime number $p$. By Lemma 3.6, \newline $O(x) \mid |C_G(x)| \mid deg(x)$, so $O(x) \in \{1, p\}$. $x$ is a vertex in $\widehat{{\mathcal S}_{G}}$ and therefore $x \notin Sol(G) \Rightarrow x \neq e_G \Rightarrow O(x) > 1 \Rightarrow O(x) = p = deg(x)$, a contradiction to Lemma 3.8. 
\end{proof}

\newpage
\noindent

\begin{lem}
Let $G$ be a non-solvable group. Then $\bigtriangleup(\widehat{{\mathcal S}_{G}}) < n-1$.
\end{lem}

\begin{proof}
There are $n = |G|-|Sol(G)|$ vertices in $\widehat{{\mathcal S}_{G}}$ and assume there exists a vertex $x$ such that $deg(x) = n-1$. Then we get 
\[
|G| - |Sol_G(x)| = |G| - |Sol(G)| - 1 \Rightarrow |Sol_G(x)| = |Sol(G)| + 1
\]
By Lemma 2.8, 
$|Sol(G)| \mid |Sol_G(x)|$ and therefore $|Sol(G)| \mid (|Sol(G)| + 1) \Rightarrow |Sol(G)| \mid 1 \Rightarrow |Sol(G)| = 1 \Rightarrow |Sol_G(x)| = 2 \Rightarrow Sol_G(x) = \{x, e_G\}$ with $O(x) = 2$. If $x$ is the only element in $G$ with $O(x) = 2$, then $x = gxg^{-1}$ $\forall g \in G \Rightarrow x \in Z(G) \subseteq Sol(G)$, a contradiction to $|Sol(G)| = 1$. Thus, there exists $y \in G$, $y \neq x$, with $O(y) = 2$. By Lemma 2.19, \newline $\langle x,y \rangle$ is solvable $\Rightarrow y \in Sol_G(x) \Rightarrow y = e_G$, a contradiction to $O(y) = 2$. Therefore every vertex $x$ in $\widehat{{\mathcal S}_{G}}$ satisfies $deg(x) < n-1 \Rightarrow \bigtriangleup(\widehat{{\mathcal S}_{G}}) < n-1$.
\end{proof}

\begin{lem}
Let $G$ be a non-solvable group. Then $\delta(\widehat{{\mathcal S}_{G}}) > 5$.
\end{lem}

\begin{proof}
Let $x$ be a vertex in $\widehat{{\mathcal S}_{G}}$. By Corollary 3.9, if $O(x) \geq 3$ then \newline $deg(x) \geq 6$. Let us consider the case $O(x) = 2$. $\widehat{{\mathcal S}_{G}}$ is connected, so there exists a vertex $y \neq x$ such that $\langle x,y \rangle$ is not solvable. If $O(y) = 2$ then by Lemma 2.19 $\langle x,y \rangle$ is solvable, a contradiction. Thus $O(y) > 2$ and the vertices $\{ y, y^{-1}, xy, yx, y^{-1}x \}$ are pairwise distinct: Indeed, if $y = y^{-1}$ or $yx = y^{-1}x$ then $O(y) = 2$, a contradiction. If one of $y, y^{-1}$ is equal to one of $xy, yx, y^{-1}x$, then $x \in \langle y \rangle \Rightarrow \langle x,y \rangle$ is solvable, a contradiction. If $xy = yx$ then $\langle x,y \rangle$ is abelian and therefore solvable, a contradiction. If $xy = y^{-1}x$ then $O(xy) = 2 \Rightarrow \langle x,y \rangle = \langle x,xy \rangle$ is solvable, a contradiction. \newline Every vertex $a \in \{ y, y^{-1}, xy, yx, y^{-1}x \}$ satisfies $\langle x,a \rangle = \langle x,y \rangle$ is not solvable $\Rightarrow deg(x) \geq 5$. By Lemma 3.10 $deg(x) \neq 5 \Rightarrow deg(x) > 5$.
\end{proof}

\begin{lem}
Let $G$ be a non-solvable group. Then $\widehat{{\mathcal S}_{G}}$ contains $K_{4,4}$ as a subgraph.
\end{lem}

\begin{proof}
Let $|G| = \prod\limits_{i = 1}^{m} p_i^{k_i}$ where $p_1 < p_2 < ... < p_m$ are prime numbers and $k_1, k_2, ..., k_m \in \N $. If $m \leq 2$ then by Burnside's theorem $G$ is solvable, a contradiction. So $m \geq 3$. If $Sol(G)$ contains all $p_i-Sylow$ subgroups for  $i \geq 3$, then $|G/Sol(G)| = p_1^{l_1}p_2^{l_2}$ for appropriate $l_1, l_2 \in \N \cup \{0\} \Rightarrow G/Sol(G)$ is solvable $\Rightarrow G$ is solvable, a contradiction. Thus, there exists $x \in G \setminus Sol(G)$ within a $p_i-Sylow$ subgroup for some $i \geq 3$. $p_3 \geq 5 \Rightarrow O(x) = p^t$ for some prime number $p \geq 5$ and $t \in \N$. Let us denote
\[
A = \{a \in \N \mid a < O(x), \text{ } gcd(a, O(x)) = 1 \} = \{ a_1, a_2, ..., a_n \}
\]
\noindent
Since $O(x) \geq 5$, $|A| = n \geq 4$. The vertex $x \in G \setminus Sol(G)$ has at least one neighbour, say $y$. Denote
\[
U = \{ x^{a_1}, x^{a_2}, ..., x^{a_n} \} \text{ , } V = \{ x^{a_1}y, x^{a_2}y, ..., x^{a_n}y \}
\]
If $U \cap V \neq \emptyset$, then $\exists 1 \leq i,j \leq n : x^{a_i} = x^{a_j}y \Rightarrow y \in \langle x \rangle \Rightarrow \langle x,y \rangle$ is solvable, a contradiction. Thus $U \cap V = \emptyset$. Since $gcd(a_i, O(x)) = 1$ $\forall 1 \leq i \leq n$, for $x^{a_i} \in U$ and $x^{a_j} \in V$ we get $\langle x^{a_i}, x^{a_j}y \rangle = \langle x, x^{a_j}y \rangle = \langle x, y \rangle$, which is not solvable. $|U| = |V| = n$, otherwise $x^{a_i} =  x^{a_j}$ or $x^{a_i}y =  x^{a_j}y$ for some $i,j$, a contradiction to $O(x)$ definition. Thus, $|U|,|V| = n \geq 4$. Every $u \in U$ is connected by edge to every $v \in V \Rightarrow \widehat{{\mathcal S}_{G}}$ contains $K_{4,4}$ as a subgraph.
\end{proof}

\begin{cor}
$\widehat{{\mathcal S}_{G}}$ is not planar.
\end{cor}

\begin{proof}
Combine Lemma 3.13 with Kuratowski's theorem (actually, Lemma 3.12 also implies that $\widehat{{\mathcal S}_{G}}$ is not planar. See~\cite{DBW}, 6.1.8). 
\end{proof}

\begin{prop}
Let $G$ be a group and ${\mathcal S}_{G}$ be its non-solvable graph. Then $|G| \mid \sum\limits_{x \in G}^{}deg(x)$.
\end{prop}

\begin{proof}
$|Sol_G(x)| + deg(x) = |G|$ $\forall x \in G \Rightarrow \sum\limits_{x \in G}^{}|Sol_G(x)|$ $+ \sum\limits_{x \in G}^{}deg(x) = |G|^{2}$. By Lemma 2.16, $|G|$ divides $\sum\limits_{x \in G}^{}|Sol_G(x)| \Rightarrow |G|$ divides $\sum\limits_{x \in G}^{}deg(x)$.
\end{proof}

\begin{prop}
Let $G$ be a group and ${\mathcal S}_{G}$ be its non-solvable graph. Then $|\{deg(x) \mid x \in G \}| \neq 2$.
\end{prop}

\begin{proof}
We use induction on $|G|$. \newline If $|G|=1$ then $G = \{e_G\}$, so $|\{deg(x) \mid x \in G \}| = 1 \neq 2$. Let's assume that the proposition holds for every group $G$ with $|G| < m$, for a certain $m \in \N$. Let $G$ be a group with $|G| = m$. If $G$ is solvable then $deg(x) = 0$ $\forall x \in G \Rightarrow |\{deg(x) \mid x \in G \}| = 1 \neq 2$. Thus, we can assume $G$ is a non-solvable group. 
\newpage \noindent
If $|Sol(G)| > 1$ then $|G/Sol(G)| < m$ and therefore 
\[
|\{deg(xSol(G)) \mid xSol(G) \in G/Sol(G) \}| \neq 2
\]
in the non-solvable graph of $G/Sol(G)$. By Lemma 2.11, 
\[
|Sol_{G/Sol(G)}(xSol(G))| = \frac{|Sol_G(x)|}{|Sol(G)|} \Rightarrow 
\]
\[
|G/Sol(G)| - deg(xSol(G)) = \frac{(|G|-deg(x))}{|Sol(G)|} 
\]
The linear relation between $deg(xSol(G))$ and $deg(x)$ implies that 
\[
|\{deg(x) \mid x \in G \}| = |\{deg(xSol(G)) \mid xSol(G) \in G/Sol(G) \}| \Rightarrow
\]
\[
|\{deg(x) \mid x \in G \}| \neq 2
\]
We are left with the case $|Sol(G)| = 1$. Assume $|\{deg(x) \mid x \in G \}| = 2$. Since $deg(e_G) = 0$ there exists $d \in \N$ such that $\{deg(x) \mid x \in G \} = \{0, d\}$. Since $deg(x) = 0 \Leftrightarrow x \in Sol(G)$ we get that $deg(x) = d$ $\forall e_G \neq x \in G$. Therefore 
\[
\sum\limits_{x \in G}^{}deg(x) = deg(e_G) \text{ } + \sum\limits_{e_G \neq x \in G}^{}deg(x) = (|G|-1)d
\]
By Proposition 3.15,
\[
|G| \mid \sum\limits_{x \in G}^{}deg(x) \Rightarrow |G| \mid (|G|-1)d
\]
$|G| > 1$ and $|G|, |G|-1$ are coprime, so $|G|$ divides $d \Rightarrow d \geq |G|$, a contradiction. 
\end{proof}

\begin{cor}
Let $G$ be a non-solvable group. Then $\widehat{{\mathcal S}_{G}}$ is irregular.
\end{cor}

\begin{proof}
If $\widehat{{\mathcal S}_{G}}$ is regular, then there exists $d \in \N$ such that
\[
\{deg(x) \mid x \in G \setminus Sol(G) \} = \{d \} \Rightarrow \{deg(x) \mid x \in G \} = \{0, d\} \Rightarrow 
\]
\[
|\{deg(x) \mid x \in G \}| = 2
\]
and we get a contradiction to the last proposition.
\end{proof}

\begin{rem}
Proposition 3.16 and Corollary 3.17 can be converted naturally to $\widehat{{\mathcal N}_{G}}$ terms. Thus, we presented an alternative, simpler proof, of prof. Abdollahi's proof for the regularity of $\widehat{{\mathcal N}_{G}}$ (~\cite{NNGG} 7.1). 
\end{rem}

\pagebreak

\end{document}